\documentclass[12pt]{amsart}

\usepackage[margin=2.6cm]{geometry}
\usepackage{amscd}
\usepackage{amssymb}
\usepackage{amsmath}
\usepackage{amsthm}
\usepackage{enumerate}
\usepackage{tikz}
\usepackage{wrapfig, framed, caption}
\usepackage{float}
\usetikzlibrary{arrows}
\usepackage[font=small,labelfont=bf]{caption}
\usepackage{xcolor}
\usepackage{mathtools}
\usepackage[colorlinks=true, linkcolor=blue, citecolor=blue,
pagebackref=true]{hyperref}
\usepackage{fouriernc}
\usepackage{ulem}

\newtheorem{theorem}{Theorem}
\newtheorem*{theorem*}{Theorem}

\newtheorem*{theoremY*}{Theorem Y}

\newtheorem*{theoremAB*}{Theorem AB}

\newtheorem*{linearformsmtp*}{Mass transference principle for linear forms}
\newtheorem{corollary}{Corollary}
\newtheorem*{corollary*}{Corollary}
\newtheorem{proposition}{Proposition}
\newtheorem{lemma}{Lemma}

\newtheorem*{claim*}{Claim}

\theoremstyle{definition}
\newtheorem{definition}{Definition}
\theoremstyle{remark}
\newtheorem{remark}{Remark}
\newtheorem*{remark*}{Remark}

\newcommand{\bp}{\mathbf{p}}

\renewcommand{\Bbb}[1]{\mathbb{#1}}
\newcommand{\N}{{\Bbb N}}         % natural numbers

\newcommand{\Q}{{\Bbb Q}}         % rational numbers
\newcommand{\R}{{\Bbb R}}        % real numbers
\newcommand{\Z}{{\Bbb Z}}         % integer numbers

\newcommand{\cH}{{\mathcal H}}

\newcommand{\cK}{\mathcal{K}}

\newcommand{\mb}{\mathbf{b}}

\newcommand{\Bad}{\mathbf{Bad}}

\DeclareMathOperator{\dimh}{\dim_H}

\newcommand{\bq}{\textbf{q}}

\numberwithin{equation}{section}
\title{Weighted Twisted Inhomogeneous Diophantine Approximation}

\author[Mumtaz Hussain]{Mumtaz Hussain}
\address{Mumtaz Hussain,  Department of Mathematical and Physical Sciences,  La Trobe University, Bendigo 3552, Australia. }
\email{m.hussain@latrobe.edu.au}

\author{ Benjamin Ward}
\address{Ben Ward,  Department of Mathematical and Physical Sciences,  La Trobe University, Bendigo 3552, Australia. }
\email{Ben.Ward@latrobe.edu.au}
\date{\today}
\date{\today}

\begin{document}
\frenchspacing

\begin{abstract}
 We prove a multidimensional weighted analogue of the well-known theorem of Kurzweil (1955) in the metric theory of inhomogeneous Diophantine approximation. Let $\sum_{i=1}^{m}\alpha_{i}=~m$ and $|\cdot|_{\alpha}=\max_{1\leq i \leq m}|\cdot|^{1/ \alpha_{i}}$. Given an $n$-tuple of monotonically decreasing univariable functions $\Psi=~(\psi_{1},\dots,\psi_{n})$ with $\psi_{i}:\R_{+} \to \R_{+}$ such that each $\psi_i(r)\to 0$ as $r\to \infty$ and fixed $A\in\R^{n\times m}$ define
 \begin{equation*}
     W_{A}(\Psi):=\left\{ \mb \in [0,1]^{n}: \begin{array}{l} |A_{i}\cdot\bq -b_{i}-p_{i}|<\psi_{i}(|\bq|_{\alpha}) \quad (1 \leq i \leq n), \\[1ex]
     \text{for infinitely many } (\bp,\bq)\in\Z^{n}\times (\Z^{m}\backslash\{0\})\end{array}
     \right\}. 
 \end{equation*} 
 We prove that the set $ W_{A}(\Psi)$ has zero-full Lebesgue measure under convergent-divergent sum conditions with some mild assumptions on $A$ and the approximating functions $\Psi$. We also prove the Hausdorff dimension results for this set.  Along with some geometric arguments, the main ingredients are the weighted ubiquity and weighted mass transference principle introduced recently by Kleinbock \& Wang [Adv. Math. 428 (2023), Paper No. 109154],  and Wang \& Wu [Math. Ann. 381 (2021), no. 1-2, 243–317] respectively.

\end{abstract}

\maketitle

\section{Introduction}

For a fixed $\xi\in\R$ consider the sequence $(\{\xi q\})_{q\in\N}$, where $\{x\}$ denotes the fractional part of $x\in\R$. When $\xi \in \Q$ the sequence is periodic, but for $\xi\in \R\backslash\Q$ the sequence is dense on the unit interval. In 1901 Minkowski proved that for any irrational $\xi\in\R$ the inequality
\begin{equation*} \label{mink-statement}
    |\xi q - b -p|<\frac{1}{4q}
\end{equation*}
has infinitely many solutions $(q,p)\in\N\times \Z$ for any $b \not \in \Z + \xi\Z$ \cite{Minkowski1900}. The constant $\frac{1}{4}$ was later made optimal by Khintchine \cite{Khintchine1946}. 
Note that rather than the classical setting of Diophantine approximation where one approximates the space $[0,1]$ by rational points $\tfrac{p}{q}\in\Q$, here we are approximating the space $[0,1]$ by the points $(\{\xi q\})_{q\in\N}$, hence the name ``twisted'' inhomogeneous Diophantine approximation. Naturally, this requires knowledge of the distribution of the points $( \{\xi q\} )_{q\in\N}$, and so the choice of $\xi\in\R$ heavily affects that rate of approximation. \par 
A natural question to ask is how many $b\in\R$ satisfy the above inequality when the right hand side is replaced by some general decreasing function $\psi: \N \to \R_{+}$. That is, how large is the set
\begin{equation*}
    W_{\xi}(\psi):=\left\{ b\in [0,1] : |\xi q -b -p|<\psi(q) \,\,\text{ for infinitely many}\,\, (q,p)\in\N\times\Z \right\}.
\end{equation*}

In 1955, Kurzweil made the following contribution to answering the above question. Recall that $\xi$ is badly approximable if there exists a constant $c(\xi)>0$ such that
\begin{equation*}
    |\xi q -p| \geq \frac{c(\xi)}{q} \quad \text{ for all } (q,p)\in\N\times \Z.
\end{equation*}
Throughout, let $\lambda_{d}(A)$ denotes the $d$-dimensional Lebesgue measure of a set $A\subseteq\R^d$.
\begin{theorem*}[{\cite{Kur55}}]\label{kurzweil}
    Let $\psi:\N\to \R_{+}$ be a non-increasing function and $\xi\in\R\backslash\Q$. Then, $\lambda_{1}(W_{\xi}(\psi))\in\{0,1\}$. Furthermore, for $\xi$ badly approximable
    \begin{equation*}
        \lambda_{1}(W_{\xi}(\psi))=\begin{cases}
            0 \quad \text{\rm if } \quad \sum\limits_{r\in\N} \psi(r)<\infty\, ,\\[2ex]
            1 \quad \text{\rm if } \quad \sum\limits_{r\in\N} \psi(r)=\infty\, .
        \end{cases}
    \end{equation*}
\end{theorem*}
Heuristically, it makes sense to consider $\xi$ badly approximable. Note that, by the Three Distance Theorem (see for example \cite{Ravenstein1988}), the sequence $(\{\xi q\})_{q\in\N}$ does not cluster in any particular region of $[0,1]$, and so are reasonably good at approximating the unit interval. Theorem~\ref{kurzweil} has since been refined to apply to more general $\xi \in \R$. In particular, Fuchs and Kim improved on the statement by considering the principle convergents of $\xi$, see \cite[Theorem 1.2]{FuchsKim} and \cite{Simmons15} for further details and contributions to this theory. \par
 When the Lebesgue measure of the set is null, for example, all functions of the form $\psi_{\tau}(q)=q^{-\tau}$ with $\tau>1$ have $\lambda_{1}(W_{\xi}(\psi_{\tau}))=0$, one can ask for a refined statement to distinguish between the null sets. In this regard, Bugeaud \cite[Theorem 1]{Bugeaud2003}, and Schmeling and Troubetzkoy \cite[Theorem 3.2]{SchmTrot03} independently proved the following result.
\begin{theorem*}[{\cite{Bugeaud2003,SchmTrot03}}]
For any $\xi\in\R\backslash\Q$ and $\tau>1$
\begin{equation*}
    \dimh W_{\xi}(\psi_{\tau})=\frac{1}{\tau},
\end{equation*}
where $\dimh$ denotes the Hausdorff dimension.
\end{theorem*}
We should remark that one of the first results in this direction was proven by Bernik and Dodson \cite[p.105]{BernikDodson} who proved the above result for $\lambda_{1}$-almost all $\xi\in[0,1]$. Furthermore, the above theorem has since been generalised to a range of more general approximation functions $\psi$ \cite[Theorem 3]{FanWu06} and a restricted set of $\xi\in\R$, where $\tau$ in the dimension is replaced by the lower order at infinity of the function $\psi$. Perhaps more importantly, it was also shown that there exists $\xi\in\R$ where the expected dimension result is false \cite[Theorem2]{FanWu06}. \par

\subsection{Higher dimensional twisted inhomogeneous approximation}

The above setup can be readily extended to higher dimensions. Throughout suppose $v=(v_{1},\dots , v_{n})\in~\R^{n}_{+}$ and $\alpha=(\alpha_{1},\dots, \alpha_{m})\in\R^{m}_{+}$ are vectors satisfying
 \begin{equation*}
     \sum_{i=1}^{n}v_{i}=n\, , \quad \sum_{i=1}^{m}\alpha_{i}=m,
 \end{equation*}
 and let
 \begin{equation*}
     |\cdot|_{v}=\max_{1\leq i \leq n}|\cdot|^{1/ v_{i}} \, , \quad |\cdot|_{\alpha}=\max_{1\leq i \leq m}|\cdot|^{1/ \alpha_{i}} \, .
 \end{equation*}
 Let $\R^{n\times m}$ denote the set of $n\times m$ matrices with real number entries and fix some $A\in\R^{n\times m}$. Given an $n$-tuple of monotonic decreasing functions $\Psi=(\psi_{1},\dots,\psi_{n})$ with $\psi_{i}:\R_{+} \to \R_{+}$, we say $\mb=(b_{1},\dots,b_{n}) \in [0,1]^{n}$ is \textit{$\Psi$-approximable for $A$} if there exists infinitely many $(\bq,\bp)=(q_{1},\dots, q_{m},p_{1},\dots,p_{n})\in(\Z^{m}\backslash\{0\})\times \Z^{n}$ solving
 \begin{equation} \label{deftwisted}
     |A_{i}\cdot\bq -b_{i}-p_{i}|<\psi_{i}(|\bq|_{\alpha}) \quad (1 \leq i \leq n),
 \end{equation}
 where $A_{i}$ denotes the $i$th row of $A$. Denote by $W_{A}(\Psi)$ the set of such $\Psi$-approximable vectors for $A$, that is
 \begin{equation*}
     W_{A}(\Psi):=\left\{ \mb \in [0,1]^{n} : \text{ \eqref{deftwisted} is solved for infinitely many } (\bq,\bp)\in(\Z^{m}\backslash\{0\})\times \Z^{n} \right\}. 
 \end{equation*} 

 One should note that the one dimensional Lebesgue measure results presented above, that is the theorems of Kurzweil, Fuchs and Kim, rely in parts on the theory of continued fractions, and so higher dimensional analogues do not readily follow. Similarly, the dimension theory result of Bugeaud, Schmeling, and Troubetzkoy, uses the Three Distance Theorem, again a result strongest in the one dimensional setting. \par 
 In generalising the technique presented in \cite{Kim23} to the weighted ubiquitous setting we are able to prove higher dimensional weighted analogues of the classical results. In order to state our results we need the following definitions.

We say, $A\in\R^{n\times m}$ is $(v,\alpha)$-singular if for any $\varepsilon>0$ and for all sufficiently large $N\geq 1$, there exists $(\bq,\bp)\in\Z^{m+n}$ solving the inequalities
  \begin{equation} \label{system1}
      \begin{cases}
          |A_{i}\cdot\bq -p_{i}|<\varepsilon N^{-v_{i}\frac{m}{n}} \quad (1\leq i \leq n), \\
          0<|\bq|_{\alpha}<N\, .
      \end{cases}
  \end{equation}
  Let $Sing_{\alpha}(v)$ denote the set of all $(v,\alpha)$-singular matrices. That is,
  \begin{equation*}
      Sing_{\alpha}(v)=\left\{ A \in \R^{n\times m}: \lim_{N\to \infty}\left( N \min_{0<|\bq|_{\alpha}<N} \max_{1\leq i \leq n} |A_{i}\cdot\bq -p_{i}|^{\frac{n}{mv_{i}}}\right) =0 \right\}.
  \end{equation*}

We also define the set of $(v,\alpha)$-badly approximable points as
 \begin{equation*}
     \Bad_{\alpha}(v):=\left\{A \in \R^{n\times m}: \liminf_{|\bq|_{\alpha}\to \infty} \left( \max_{1\leq i \leq n} |\bq|_{\alpha}|A_{i}\cdot \bq -p_{i}|^{\frac{n}{m v_{i}}} \right)>0 \right\}.
 \end{equation*}
  
Lastly, define the set
\begin{equation*}
    L_{\alpha}(v,A,\varepsilon):=\left\{\ell\in \N : \begin{cases} |A_{i}\cdot \bq - p_{i}|<\varepsilon 2^{-\ell v_{i}\frac{m}{n}} \quad (1\leq i \leq n), \\
    0<|\bq|_{\alpha}<2^{\ell}
    \end{cases} \text{ has no solution } (\bq,\bp)\in\Z^{m+n} \right\}.
\end{equation*}
Given the above definitions, we are able to state our results.
\begin{theorem} \label{twisted statement}
    Let $A\in Sing_{\alpha}(v)^{c}$ and let $\Psi=(\psi_{1},\dots, \psi_{n})$ be an $n$-tuple of monotonic approximation functions with each
    \begin{equation*}
        \psi_{i}(r)\ll r^{-v_{i}\frac{m}{n}} \quad (1\leq i \leq n)
    \end{equation*}
    for all $r\in\R_{+}$ with the implied constants independent of $r$. Then
    \begin{equation*}
        \lambda_{n}\left(W_{A}(\Psi) \right)=1 \quad \text{\rm if } \sum_{\ell \in L_{\alpha}(v,A,\varepsilon)} 2^{m\ell }\prod_{i=1}^{n}\psi_{i}\left(2^{\ell}\right)=\infty.
    \end{equation*}
\end{theorem}

\begin{corollary} \label{bad_corollary}
     Let $A\in Bad_{\alpha}(v)$ and let $\Psi=(\psi_{1},\dots, \psi_{n})$ be an $n$-tuple of monotonic approximation functions with each
    \begin{equation*}
        \psi_{i}(r)\ll r^{-v_{i}\frac{m}{n}} \quad (1\leq i \leq n)
    \end{equation*}
    for all $r\in\R_{+}$ with the implied constants independent of $r$. Then
    \begin{equation*}
        \lambda_{n}\left(W_{A}(\Psi) \right)=\begin{cases}
        0 \quad \text{\rm if } \sum\limits_{r\in\N} r^{m-1}\prod_{i=1}^{n}\psi_{i}\left(r\right)<\infty \, , \\[2ex]
            1 \quad \text{\rm if } \sum\limits_{r\in\N} r^{m-1}\prod_{i=1}^{n}\psi_{i}\left(r\right)=\infty\, .
        \end{cases}
    \end{equation*}
\end{corollary}

\begin{proof}
    This easily follows from Theorem~\ref{twisted statement} on the observation that $\Bad_{\alpha}(v) \subset Sing_{\alpha}(v)^{c}$ and that $L_{\alpha}(v,A,\varepsilon)\supset \N_{\geq k}$ for some $\varepsilon>0$ and $k \in \N$ when $A\in \Bad_{\alpha}(v)$ (see Lemma~\ref{v-bad return sequence} in \S 2). The convergence case will be proven at the end of \S \ref{twisted ubiquity and proof}. 
\end{proof}

In \cite{Kim23} Kim proves the analogue of Theorem~\ref{twisted statement} in the case where $\alpha=(1,\dots,1)$, $v=(1,\dots,1)$ and $\Psi=(\psi,\dots,\psi)$. It is from this paper that we draw our inspirations to prove our weighted analogues, see \cite[Theorem 1.3, 
Corollary 1.6]{Kim23} for more details. It should be noted, in particular, that Kim's version of our result is in terms of the Hausdorff $s$-measure for $0\leq s\leq n$, and so, in particular, they prove that, for $A\in \Bad_{(1,\dots,1)}(1,\dots,1)$ and $\Psi(q)=(q^{-\tau}, \dots, q^{-\tau})$ with $\tau>\frac{m}{n}$, that
\begin{equation*}
    \dimh W_{A}(\Psi)=\frac{m}{\tau}.
\end{equation*}

 Given our setup, we are able to prove the weighted analogue of their result as follows.  
\begin{theorem} \label{dim result}
    Let $v=(1,\dots, 1)$. Let $A\in Bad_{\alpha}(v)$ and let $\Psi=(\psi_{1}, \dots, \psi_{n})$ be an $n$-tuple of functions with each
    \begin{equation*}
        \psi_{i}(r)=r^{-\tau_{i}} \quad \text{ and } \quad  \tau_{i} > \frac{m}{n} \quad (1\leq i \leq n).
    \end{equation*}
    Then,
    \begin{equation*}
        \dimh W_{A}(\Psi)= \min_{1\leq i \leq n}\left\{  \frac{m+\sum_{i:\tau_{j}>\tau_{i}}(\tau_{j}-\tau_{i})}{\tau_{j}}\right\}=s.
    \end{equation*}
    Furthermore, for any ball $B\subset [0,1]^{n}$ we have that
    \begin{equation*}
        \cH^{s}\left(B\cap W_{A}(\Psi)\right)=\infty\, .
    \end{equation*}
\end{theorem}

\begin{theorem} \label{dim result 2}
     Fix $n=2$ and let $v=(v_{1},v_{2})$. Let $A\in Bad_{\alpha}(v)$ and let $\Psi=(\psi_{1},\psi_{2})$ be a pair of functions with each
     \begin{equation*} \label{normal condition}
         \psi_{i}(r)=r^{-\tau_{i}}\, , \quad \text{ with } \quad \tau_{i}\geq v_{i}\frac{m}{2} \quad (i=1,2).
     \end{equation*}
     Furthermore, suppose that
     \begin{equation} \label{annoying condition}
         \min\left\{\frac{\min\{v_{1},v_{2}\} m}{\min\{\tau_{1},\tau_{2}\}n},\frac{\min\{v_{1},v_{2}\}}{\max\{v_{1},v_{2}\}} \right\}\geq \frac{m-\min\{\tau_{1},\tau_{2}\}}{\max\{\tau_{1},\tau_{2}\}}\, .
     \end{equation}
     Then,
     \begin{equation*}
         \dimh W_{A}(\Psi)=\frac{m+\max\{\tau_{1},\tau_{2}\}-\min\{\tau_{1},\tau_{2}\}}{\max\{\tau_{1},\tau_{2}\}}.
     \end{equation*}
 \end{theorem}
 \begin{remark}\rm
 The difference between Theorem~\ref{dim result} and Theorem~\ref{dim result 2} is subtle. When each of the approximation functions $\psi_{i}$ decreases at a faster rate than the Dirichlet exponent of approximation (in this setting the exponent is $\frac{m}{n}$) then Theorem~\ref{dim result} is optimal. However, when this is not the case, Theorem~\ref{dim result 2} is required. \par
 Condition \eqref{annoying condition} seems unnecessary but it is required for our method to work in the lower bound. In particular, observe that if $\min\{\tau_{1},\tau_{2}\}\geq m$, that is the approximation functions are decreasing fast enough, then \eqref{annoying condition} is automatically satisfied.
\end{remark}

It should be noted that a set of particular interest within the setting of twisted inhomogeneous approximation is the set of $\xi$-badly approximable points, which are generally the set of points $b\in[0,1]$ such that \eqref{mink-statement} becomes false when the right hand side is multiplied by some arbitrarily small constant. In this article, we do not consider such a set (although we can deduce a metric result on weighted higher dimensional analogue of $\xi$-badly approximable points from Theorem~\ref{twisted statement}). For more details on the metric properties of these sets, we refer the reader to \cite{BengMosch2017,Harrap12, MosHar, Kim07, Kim14,   Ramirez18, Tseng08, LRSY} and references within. This is a particularly active area of research, as far as we are aware there are at least two forthcoming papers in this area of research. In \cite{BDGW2} the metric results of higher dimensional analogues of $\xi$-badly approximable points are studied in detail and generalised to the $S$-arithmetic setting, and in \cite{MRS} the measure of $\xi$-badly approximable points when shifted by some constant is proven to be null for any algebraic measure on the $m$-dimensional torus (see \cite[Theorem 1.6]{LRSY}). In both these settings $\xi$ satisfies certain properties. \par 

The rest of the paper is laid out as follows. In the next section, we recall and prove a few basic properties on the set of $(v,\alpha)$-non-singular matrices. We also recall the setup for weighted ubiquitous systems as introduced in \cite{KW23, WW19}. Lastly, in \S 3, we give the proofs of our main results.

\noindent{\bf Acknowledgments:} The research of both authors is supported by the Australian Research Council discovery project 200100994. We would also like to thank Victor Beresnevich for many useful comments on an earlier draft.

 \section{Preliminaries and Auxiliary Results}

 \subsection{Properties of singular and badly approximable matrices}
 
 The following observation was made in \cite{Kim23} in the case $v=(1,\dots,1)$ and $\alpha=(1,\dots,1)$. As we will require such an observation in our results we state and prove the following easy lemma.
    \begin{lemma} \label{v-sing equivalency}
    $A\in Sing_{\alpha}(v)$ if and only if for any $\varepsilon>0$ the set
\begin{equation*}
    L_{\alpha}(v,A,\varepsilon):=\left\{\ell\in \N : \begin{cases} |A_{i}\cdot \bq - p_{i}|<\varepsilon 2^{-\ell v_{i}\frac{m}{n}} \quad (1\leq i \leq n), \\
    0<|\bq|_{\alpha}<2^{\ell}
    \end{cases} \text{\rm has no solution } (\bq,\bp)\in\Z^{m+n} \right\}
\end{equation*}
is finite.
 \end{lemma} 

 \begin{remark} \rm 
 Observe that this result naturally implies that if $A$ is $(v,\alpha)$- non-singular, that is, $A\in Sing_{\alpha}(v)^{c}=\R^{n\times m}\backslash Sing_{\alpha}(v)$, then $L_{\alpha}(v,A,\varepsilon)$ is unbounded. In particular this means the summation appearing in Theorem~\ref{twisted statement} is infinite.
 \end{remark}
\begin{proof}
    The forward implication that $A\in Sing_{\alpha}(v)$ implies $L_{\alpha}(v,A,\varepsilon)$ is immediate by the definition and setting $N=2^{\ell}$. To see the reverse implication note that if $L_{\alpha}\left(v,A,\tfrac{\varepsilon}{2^{\max v_{i}}}\right)$ is finite, say $L_{\alpha}\left(v,A,\tfrac{\varepsilon}{2^{\max v_{i}}}\right) \subset \{1,\dots,k\}$, then for all $\ell >k$ and any $2^{\ell}\leq N < 2^{\ell+1}$ we have that
    \begin{equation*}
        \begin{cases} |A_{i}\cdot \bq - p_{i}|<\frac{\varepsilon}{2^{\frac{m}{n}\max v_{i}}} 2^{-\ell v_{i}\frac{m}{n}} \leq \varepsilon 2^{-(\ell+1)v_{i}\frac{m}{n}} < \varepsilon N^{-v_{i}\frac{m}{n}} \quad (1\leq i \leq n), \\
    0<|\bq|_{\alpha}<2^{\ell}\leq N
    \end{cases}
    \end{equation*}
    has solution $(\bq,\bp)\in\Z^{m+n}$. Since this is true for any choice of $\varepsilon>0$ we have that $A\in Sing_{\alpha}(v)$.
\end{proof}

As stated in \cite{Kim23} in the case of $v=(1,\dots,1)$ and $\alpha=(1,\dots,1)$ we have the following result.
\begin{lemma} \label{v-bad return sequence}
$A\in \Bad_{\alpha}(v)$ if and only if for some $\varepsilon>0$ there exists $k\in\N$ such that $L_{\alpha}(v,A,\varepsilon)\supset \N_{\geq k}$.
\end{lemma}
 \begin{proof}
     If $A\in\Bad_{\alpha}(v)$ then there exists some $c(A)>0$ such that
     \begin{equation*}
          \max_{1\leq i \leq n} |A_{i}\cdot \bq -p_{i}|^{1/v_{i}} >c(A)|\bq|_{\alpha}^{-\frac{m}{n}}
     \end{equation*}
     for all sufficiently large $\bq\in\Z^{m}\backslash\{0\}$. Without loss of generality assume this is true for all $\bq\in \Z^{m}$ such that $|\bq|_{\alpha}\geq 2^{t}$. Hence for any $\ell>t$ and any $|\bq|_{\alpha}<2^{\ell}$ we have that
     \begin{equation*}
         \max_{1\leq i \leq n} |A_{i}\cdot \bq -p_{i}|^{1/v_{i}} >c(A)|\bq|_{\alpha}^{-\frac{m}{n}}>c(A)2^{-\ell\frac{m}{n}}.
     \end{equation*}
     Thus $L_{\alpha}\left(v,A,c(A)^{\min v_{i}}\right) \supset \N_{>t}$. \par 
     For the reverse implication, we have that for all $\ell>k\in\N$
     \begin{equation*}
         \begin{cases}
             \max_{1\leq i \leq n}|A_{i}\cdot\bq -p_{i}|^{1/v_{i}}>\varepsilon 2^{-\ell\frac{m}{n}}\, ,\\
             0<|\bq|_{\alpha}<2^{\ell}\, .
         \end{cases}
     \end{equation*}
     Hence for all sufficiently large $\bq\in\Z^{m}$ (all $\bq\in\Z^{m}$ such that $|\bq|_{\alpha}>2^{k}$) there exists $\ell>k$ such that $2^{\ell}<|\bq|_{\alpha}\leq 2^{\ell+1}$ so that
     \begin{equation*}
         \begin{cases}
             \max_{1\leq i \leq n}|A_{i}\cdot\bq -p_{i}|^{1/v_{i}}>\varepsilon 2^{-(\ell+1)\frac{m}{n}}=\frac{\varepsilon}{2^{\frac{m}{n}}}2^{-\ell\frac{m}{n}}> \frac{\varepsilon}{2^{\frac{m}{n}}}|\bq|_{\alpha}^{-\frac{m}{n}} ,\\
             \text{ for all } 0<|\bq|_{\alpha}<2^{\ell+1}\, .
         \end{cases}
     \end{equation*}
     This is true for all $\ell>k$, hence $A\in\Bad_{\alpha}(v)$.
 \end{proof}

\subsection{A Dirichlet-type theorem in the case of non-singular matrices}

We need the following weighted analogue of \cite[Chapter V, Theorem VI]{Cassels}, which can readily be proven from \cite[Chapter V, Theorem V]{Cassels}. For completeness, we prove the result here.

 \begin{lemma} \label{inhomogeneous transference}
     Let $A\in \R^{nm}$ and suppose there are no integer solutions $(\bq,\bp) \in \Z^{m+n}\backslash\{0\}$ to
     \begin{equation*} \label{system1}
     \left\{\begin{array}{c}
         |A\bq-\bp|_{v}<C,\\
         |\bq|_{\alpha}<N,
         \end{array}\right. \,
     \end{equation*}
     for norms $|\cdot|_{v}=\max_{1\leq i \leq n}|\cdot|^{1/v_{i}}$, $|\cdot|_{\alpha}=\max_{1\leq i \leq m} |\cdot|^{1/\alpha_{i}}$ and vectors $v=(v_{1},\dots, v_{n})\in\R^{n}_{+}$ and $\alpha=(\alpha_{1}, \dots , \alpha_{m})\in\R^{m}_{+}$ satisfying
     \begin{equation*}
         \sum_{i=1}^{n}v_{i}=n \quad \text{ and } \quad \sum_{i=1}^{m}\alpha_{i}=m\, .
     \end{equation*}
     Then for any $\textbf{b} \in \R^{n}$ there exists $(\bq,\bp) \in \Z^{m+n}$ solving
     \begin{equation*}
         \left\{\begin{array}{c}
         |A\bq -\textbf{b}-\bp|_{v}\leq c_{1} C,\\
         |\bq|_{\alpha}<c_{1} N,
         \end{array}\right.
     \end{equation*}
     for constant
     \begin{equation*}
         c_{1}=\underset{1\leq j\leq n}{\max_{1\leq i \leq m}}\left\{ \left(\frac{1}{2}\left(C^{-n}N^{-m}+1\right)\right)^{1/v_{i}}, \left(\frac{1}{2}\left(C^{-n}N^{-m}+1\right)\right)^{1/\alpha_{j}} \right\}.
     \end{equation*}
 \end{lemma}

 \begin{proof}
     Consider the system of inequalities on $n+m$ variables $(\bq,\bp)\in\Z^{m+n}$
     \begin{equation*}
         \left\{\begin{array}{c}
         |C^{-v_{i}}(A_{i}\cdot\bq -p_{i})|<1, \quad (1\leq i \leq n)\\
         |N^{-\alpha_{j}}q_{j}|<1, \quad (1\leq j \leq m).
         \end{array} \right.
     \end{equation*}
     For ease of notation let $f_{i}(\bq,\bp)=C^{-v_{i}}(A_{i}\cdot\bq -p_{i})$ for $1\leq i\leq n$ and $f_{i+m}(\bq,\bp)=N^{-\alpha_{j}}q_{j}$ for $1\leq i \leq m$. Then \eqref{system1} is equivalent to
     \begin{equation} \label{system2}
     \max_{1\leq i \leq n+m} |f_{i}(\bq,\bp)|<1\, .
     \end{equation}
     By assumption, \eqref{system2} has no integer solutions $(\bq,\bp)\in\Z^{m+n}\backslash \{0\}$. Furthermore, observe that the $(n+m)\times(n+m)$ matrix associated to the system of linear forms \eqref{system2} has determinant 
     $$C^{-\sum_{i=1}^{n}v_{i}}N^{-\sum_{i=1}^{m}\alpha_{i}}=C^{-n}N^{-m},$$
     and so by \cite[Theorem V]{Cassels} for any real number $\textbf{b}^{*}\in\R^{n+m}$ there are integer solutions to 
     \begin{equation} \label{system3}
         \max_{1\leq i \leq n+m}|f_{i}(\bq,\bp)-b^{*}_{i}|<\frac{1}{2}(C^{-n}N^{-m}+1)\, .
     \end{equation}
     Setting $\textbf{b}^{*}=(C^{-v_{1}}b_{1}, \dots , C^{-v_{n}}b_{n},0,\dots, 0)\in \R^{n+m}$ in \eqref{system3} gives us that that the system of inequalities
     \begin{equation*}
         \left\{\begin{array}{c}
         |A_{i}\cdot \bq -b_{i}-p_{i}|\leq \frac{1}{2}(C^{-n}N^{-m}+1)C^{v_{i}}, \quad (1\leq i \leq n)\\
         |q_{i}|<\frac{1}{2}(C^{-n}N^{-m}+1)N^{\alpha_{i}}, \quad (1\leq i \leq m) 
         \end{array}\right.
     \end{equation*}
     has integer solutions $(\bq,\bp)\in\Z^{n+m}\backslash\{0\}$ for any $\textbf{b}\in\R^{n}$. Rearranging in terms of the norms $|\cdot|_{v}$ and $|\cdot|_{\alpha}$ completes the proof.
 \end{proof}

 For Lemma~\ref{inhomogeneous transference} we can easily deduce the following corollary
 \begin{corollary} \label{twisted_dirichlet}
     Let $A\in Sing_{\alpha}(v)^{c}=[0,1]^{nm}\backslash Sing_{\alpha}(v)$. Then for any $\mb \in \R^{n}$ there exists some $\varepsilon>0$ such that for all $\ell \in L_{\alpha}(v,A,\varepsilon)$ the system of inequalities
     \begin{equation*}
         \begin{cases}
             |A_{i}\cdot \bq -p_{i}-b_{i}|<\varepsilon c_{2}2^{-\ell v_{i}\frac{m}{n}} \quad (1\leq i \leq n),\\
             |\bq|_{\alpha}< c_{2}2^{\ell}
         \end{cases}
     \end{equation*}
     has integer solution $(\bq,\bp) \in \Z^{m+n}$ for $c_{2}=\left(\frac{1}{2}(\varepsilon^{-n}+1)\right)^{\frac{1}{\min_{i,j}\{v_{i},\alpha_{j}\}}}$.
 \end{corollary}

 \begin{proof}
     Note that since $A\in Sing_{\alpha}(v)^{c}$, there exists $\varepsilon>0$ such that the set $L_{\alpha}(v,A,\varepsilon)$ has infinite cardinality by Lemma~\ref{v-sing equivalency}. Now take $C=\varepsilon 2^{-\ell\frac{m}{n}}$ and $N=2^{\ell}$ as in Lemma~\ref{inhomogeneous transference}.
 \end{proof}

Observe that Corollary~\ref{twisted_dirichlet} implies that $A\bq ({\rm mod} 1)$ is dense in $[0,1]^{m}$, and so, by Kronecker's Theorem (see for example \cite[Chapter III Theorem IV]{Cassels}), the subgroup $G( ^{t}A):= ^{t}A\Z^{m} + \Z^{n} \subset \R^{n}$ has maximal rank $n+m$ over $\Z$. This allows us to use the following result.
\begin{lemma}[{\cite[Proposition 3.5]{Kim23}}] \label{weyl kim}
Suppose $A\in Sing_{\alpha}(v)^{c}$. Then for any ball $B\subset [0,1]^{n}$
\begin{equation*}
    \frac{\#\{A\bq \in B : |\bq|_{\alpha}\leq N\}}{\#\{\bq \in \Z^{n}: |\bq|_{\alpha}\leq N\}}\to \lambda_{n}(B) \quad \text{ as } N\to \infty\, .
\end{equation*}
\end{lemma}
The proof of this Lemma follows in exactly the same way as \cite[Proposition 3.5]{Kim23}, the only difference being in the latter stages of the proof, where the summation on $q_{1}$ ranges over $-N^{\alpha_{1}}$ to $N^{\alpha_{1}}$ and is averaged over $N^{\alpha_{1}}$.

 \subsection{Weighted ubiquitous systems}\label{weighted ubiquity sect}
 In this section we give the definition of local ubiquity for rectangles as given in \cite{KW23}. This definition is a generalisation of ubiquity for rectangles as found in \cite{WW19}, which is in turn a generalisation of local ubiquity for balls introduced in \cite{BDV}. For brevity, we will state the results of \cite{KW23,WW19} in the special setting of $n$-dimensional real space. We will also assume that each resonant set, see definition below, is a finite collection of points and so we can omit the notion of $\kappa$-scaling. See \cite{KW23, WW19} for the full statements. \par 
  Consider the product space $(\R^{n},|\cdot|_{\infty},\lambda_{n})$, where $ |\cdot|_{\infty}=\max_{1 \leq i \leq n}|\cdot|$. For any $x \in \R^{n}$ and $r \in \R_{+}$ define the open ball
\begin{equation*}
B(x,r)=\left\{ y \in \R^{n}: |x-y|_{\infty}< r \right\}=\prod_{i=1}^{n}B_{1}(x_{i},r),
\end{equation*}
where $B_{1}$ are the usual open intervals with centre $x_{i}$ and diameter $2r$ in $\R$. Let $J$ be a countably infinite index set, and $\beta: J \to \R_{+}$, $\alpha \mapsto \beta_{\alpha}$ a positive function satisfying the condition that for any $N \in \N$
\begin{equation*}
\# \left\{ \alpha \in J: \beta_{\alpha} < N \right\} < \infty.
\end{equation*}
 Let $l_{k},u_{k}$ be two sequences in $\R_{+}$ such that $u_{k} \geq l_{k}$ with $l_{k} \to \infty$ as $k \to \infty$. Define
\begin{equation*}
J_{k}= \{ \alpha \in J: l_{k} \leq \beta_{\alpha} \leq u_{k} \}.
\end{equation*}
Let $\rho=(\rho_{1}, \dots , \rho_{n})$ be an $n$-tuple of non-increasing functions $\rho_{i}: \R_{+} \to \R_{+}$ such that each $\rho_{i}(x) \to 0$ as $x\to \infty$. For each $1 \leq i \leq n$, let $( R_{\alpha,i})_{\alpha \in J}$ be a sequence of finite collections of points in $\R$.
The family of sets $( R_\alpha)_{\alpha\in J}$ where
\begin{equation*}
	R_{\alpha}=\prod_{i=1}^{n} R_{\alpha, i}, %_{ \alpha \in J}.
\end{equation*}
for each $ \alpha \in J$, are called \textit{resonant sets}.

Define
\begin{equation*}
\Delta(R_{\alpha},\rho(r))= \prod_{i=1}^{n}  \Delta_{i}(R_{\alpha,i},\rho_{i}(r)),
\end{equation*}
where for some set $Y\subset \R$ and $b \in \R_{+}$
\begin{equation*}
\Delta_{i}(Y,b)= \bigcup_{a \in Y}B_{1}(a,b)
\end{equation*}
is the union of balls in $\R$ of radius $b$ centred at all possible points in $Y$.

The following notion of ubiquity for rectangles can be found in \cite{KW23}.

\begin{definition}[Local ubiquitous system of rectangles]\rm
Call the pair $\big((R_{\alpha})_{\alpha \in J}, \beta\big)$ {\em a local ubiquitous system of rectangles with respect to $\rho$} if there exists a constant $c>0$ such that for any ball $B \subset X$ and all sufficiently large $k \in \N$
\begin{equation*}
\lambda_{n}\left( B \cap \bigcup_{\alpha \in J_{k}}\Delta(R_{\alpha}, \rho(u_{k})) \right) \geq c m(B).
\end{equation*}
\end{definition}
For $n$-tuple of approximation functions $\Psi=(\psi_{1},\dots, \psi_{n})$ with each $\psi_{i}:\R_{+} \to \R_{+}$ define
\begin{equation*}
    W(\Psi)=\left\{x \in \R^{n} : x \in \Delta\left(R_{\alpha},\Psi(\beta_{\alpha})\right) \, \text{ for infinitely many } \alpha \in J \right\}.
\end{equation*}

The following theorem, due to Kleinbock and Wang \cite{KW23}, provides the Lebesgue measure theory on $W(\Psi)$.

\begin{theorem} \label{KW ambient measure}
    Let $0<c<1$. A function $f$ is said to be $c$-regular with respect to a sequence $\{r_{i}\}_{i\in\N}$ if $f(r_{i+1}) \leq c f(r_{i})$ for all sufficiently large $i$. Let $W(\Psi)$ be defined as above and assume that $(\{R_{\alpha}\}_{\alpha \in J}, \beta)$ is a local ubiquitous systems of rectangles with respect to $\rho$. Suppose that
    \begin{enumerate}
        \item[\rm (I)] each $\psi_{i}$ is decreasing,
        \item[\rm(II)] for each $1\leq i \leq n$,  $\psi_{i}(r) \leq \rho_{i}(r)$ for all $r\in\R_{+}$ and $\rho_{i}(r)\to 0$ as $r\to \infty$,
        \item[\rm (III)] either $\rho_{i}$ is $c$-regular on $\{u_{k}\}_{k\in\N}$ for all $1\leq i \leq n$ or $\psi_{i}$ is $c$-regular on $\{u_{k}\}_{k\in\N}$ for all $1\leq i \leq n$ for some $0<c<1$.
    \end{enumerate}
    Then
    \begin{equation*}
        \lambda_{n}(W(\Psi))=\textit{full} \quad \text{\rm if } \quad \sum_{k=1}^{\infty}\prod_{i=1}^{n}\left(\frac{\psi_{i}(u_{k})}{\rho_{i}(u_{k})}\right)\, = \infty.
    \end{equation*}
\end{theorem}
Here, by full we mean that the complement is a Lebesgue nullset. \par 
For the Hausdorff theory we have the following theorem due to Wang and Wu \cite{WW19}.

\begin{theorem} \label{MTPRR}
Let $W(\Psi)$ be defined as above and assume that $(\{R_{\alpha}\}_{\alpha \in J}, \beta)$ is a local ubiquitous systems of rectangles with respect to $\rho=(\rho^{a_{1}}, \dots ,\rho^{a_{n}})$ for some function $\rho:\R_{+} \to \R_{+}$ and $(a_{1},\dots, a_{n}) \in \R^{n}_{+}$ with $\rho(N)\to 0$ as $N\to \infty$. Then, for $\Psi=(\rho^{a_{1}+t_{1}},\dots, \rho^{a_{n}+t_{n}})$ for some $\textbf{t}=(t_{1}, \dots, t_{n}) \in \R^{n}_{+}$
\begin{equation*}
\dimh W(\Psi) \geq \min_{A_{i} \in A} \left\{ \sum_{j \in \cK_{1}} 1+ \sum_{j \in \cK_{2}} 1+ \frac{\sum_{j \in \cK_{3}}a_{j}-\sum_{j \in \cK_{2}}t_{j}}{A_{i}} \right\}=s,
\end{equation*}
where $A=\{ a_{i}, a_{i}+t_{i} , 1 \leq i \leq n \}$ and $\cK_{1},\cK_{2},\cK_{3}$ are a partition of $\{1, \dots, n\}$ defined as
\begin{equation*}
 \cK_{1}=\{ j:a_{j} \geq A_{i}\}, \quad \cK_{2}=\{j: a_{j}+t_{j} \leq A_{i} \} \backslash \cK_{1}, \quad \cK_{3}=\{1, \dots n\} \backslash (\cK_{1} \cup \cK_{2}).
 \end{equation*}
 Furthermore, for any ball $B \subset \R^{n}$
  \begin{equation*} \label{MTPRR_measure}
\cH^{s}(B \cap W(\Psi))=\cH^{s}(B).
\end{equation*}
 \end{theorem}

\section{Proof of main results}
The main results will be proven by showing that we have a weighted ubiquitous system for a certain setup. Given this statement, we can use the theorems of Kleinbock \& Wang (Theorem~\ref{KW ambient measure}), and Wang \& Wu (Theorem~\ref{MTPRR}) to obtain the divergence case of Theorem~\ref{twisted statement} and the lower bound of Theorem ~\ref{dim result}. The corresponding convergence and upper bound statements use standard covering arguments, but for completeness, we include them in their relevant sections.
\subsection{Proof of the weighted ubiquity statement} \label{twisted ubiquity and proof}

Choose any $A\in Sing_{\alpha}(v)^{c}$ and fix such a choice for the remainder of this section. Let
\begin{align*}
&J=\left\{\bq \in \Z^{m}\right\} \quad \beta:J\to \R_{+},  \bq \mapsto \beta_{\bq}=|\bq|_{\alpha}, \\
&R_{\bq,i}= \{A_{i}\cdot \bq\}  \quad R_{\bq}=\{A\bq\}\, ,
\end{align*}
where $\{X\}$ denotes the vector composing of the fractional part in each coordinate. Define the sequences
\begin{equation*}
    l_{i}=c_{3}u_{i} \text{ for constant } 0<c_{3}<1 \text{ dependent on $\varepsilon$}\, , \quad u_{i}=c_{2}2^{\ell_{i}}\\
\end{equation*}
where $\varepsilon>0$ is the real number such that the set $L_{\alpha}(v,A,\varepsilon)$ is infinite, and $\{\ell_{i}\}_{i\in\N}$ corresponds to the ordered sequence of such integers ($c_{2}$ is the constant given in Corollary~\ref{twisted_dirichlet}).

\begin{proposition}
    Let each
    \begin{equation*}
        \rho_{i}(r)=\varepsilon c_{2}^{1+v_{i}\frac{m}{n}} r^{-v_{i}\frac{m}{n}} \quad (1\leq i \leq n)\, .
    \end{equation*}
    Then for any ball $B\subset[0,1]^{n}$
    \begin{equation*}
        \lambda_{n}\left(B\cap \bigcup_{\bq \in J_{k}}\Delta\left(R_{\bq},\rho(u_{k})\right) \right) \geq \frac{1}{2}\lambda_{n}(B),
    \end{equation*}
    for all sufficiently large $k \in \N$.
\end{proposition}

 \begin{proof}
     Fix $B \subset [0,1]^{n}$. By Corollary~\ref{twisted_dirichlet} we have that for any $\ell_{j} \in L_{\alpha}(v,A,\varepsilon)$
     \begin{equation*}
         \lambda_{n}\left( B\cap \bigcup_{\bq\in\Z^{m} : |\bq|_{\alpha}\leq c_{2}2^{\ell_{k}}}\Delta\left(R_{\bq}, \left(\varepsilon c_{2}2^{-\ell_{k}v_{1}\frac{m}{n}}, \dots, \varepsilon c_{2}2^{-\ell_{k}v_{n}\frac{m}{n}}\right) \right) \right)=\lambda_{n}(B). 
     \end{equation*}
     Note that each
     \begin{equation*}
         \varepsilon c_{2}2^{-\ell_{k}v_{i}\frac{m}{n}}=\varepsilon c_{2}^{1+v_{i}\frac{m}{n}} \left(c_{2}2^{\ell_{j}}\right)^{-v_{i}\frac{m}{n}}=\rho_{i}(u_{k}) \quad (1\leq i \leq n).
     \end{equation*}
     This allows us to deduce that for any $k \in \N$
     \begin{align*}
         \lambda_{n}(B) &\leq \lambda_{n}\left( B\cap \bigcup_{\bq\in\Z^{m} : |\bq|_{\alpha}\leq l_{k}}\Delta\left(R_{\bq}, \rho(u_{k})\right) \right)+\lambda_{n}\left( B\cap \bigcup_{\bq\in\Z^{m} : \bq \in J_{k}}\Delta\left(R_{\bq}, \rho(u_{k})\right) \right)\, ,
     \end{align*}
    and so the proof is complete in showing that
    \begin{equation}\label{twisted last bit}
        \lambda_{n}\left( B\cap \bigcup_{\bq\in\Z^{m} : |\bq|_{\alpha}\leq l_{k}}\Delta\left(R_{\bq}, \rho(u_{k})\right) \right) \leq \frac{1}{2}\lambda_{n}(B).
    \end{equation}
     To show \eqref{twisted last bit} we note, by Lemma~\ref{weyl kim}, that there exists sufficiently large $k_{B}\in\N$ such that for all $k>k_{B}$
     \begin{equation*}
         \#\{A\bq \in 2B: |\bq|_{\alpha}\leq l_{k}\} \leq 2 l_{k}^{m}\lambda_{n}(B).
     \end{equation*}
     Hence, taking $k \in \N$ sufficiently large such that $\max_{1\leq i \leq n} \rho_{i}(u_{k})<r(B)$, we have that
     \begin{align*}
         \lambda_{n}\left( B\cap \bigcup_{\bq\in\Z^{m} : |\bq|_{\alpha}\leq l_{k}}\Delta\left(R_{\bq}, \rho(u_{k})\right) \right) &\leq 2l_{k}^{m}\lambda_{n}(B)2^{n}\prod_{i=1}^{n}\rho_{i}(u_{k})\\
         & =2^{n+1}\varepsilon^{n}c_{3}^{m}c_{2}^{n+m}\lambda_{n}(B)\\
         & \leq  \frac{1}{2}\lambda_{n}(B),
     \end{align*}
     where the last inequality follows since we can choose the constant $c_{3}$ so that $$c_{3}<\left(2^{-(n+2)}\varepsilon^{-n}c_{2}^{-(n+m)}\right)^{1/m}$$ independent of $B$ and $k$.
 \end{proof}

\subsection{Proof of Theorem~\ref{twisted statement}}
To prove Theorem~\ref{twisted statement} we apply Theorem~\ref{KW ambient measure}. The bulk of the proof (the weighted ubiquity statement) has been proven in the previous section. It remains to verify that the choice of $\rho$ function satisfies the conditions of Theorem~\ref{KW ambient measure} and that the summations given in Theorem~\ref{KW ambient measure} and Theorem~\ref{twisted statement} are equivalent.\par 
By the conditions of Theorem~\ref{twisted statement}, each $\psi_{i}$ is monotonically decreasing. Furthermore each $\rho_{i}(r)\geq \psi_{i}(r)$. Note that each $\psi_{i}$ may need to be multiplied by some constant to make this strictly true, but as shown in for example \cite[Lemma 5.7]{BDGW23}, this would not change the Lebesgue measure of $W_{A}(\Psi)$. Observe that each
\begin{equation*}
    \rho_{i}(u_{k+1})=\varepsilon c_{2}^{1+v_{i}\frac{m}{n}}u_{k+1}^{-v_{i}\frac{m}{n}}=\varepsilon c_{2} 2^{-\ell_{k+1}v_{i}\frac{m}{n}}\leq \varepsilon c_{2} 2^{-(\ell_{k}+1)v_{i}\frac{m}{n}} =2^{-v_{i}\frac{m}{n}}\rho_{i}(u_{k})\leq 2^{-\frac{m}{n}\min_{i}v_{i}}\rho_{i}(u_{k})
\end{equation*}
and so $\rho$ is $2^{-\frac{m}{n}\min_{i}v_{i}}$-regular. Thus, conditions $(I)-(III)$ of Theorem~\ref{KW ambient measure} are satisfied, and so 
\begin{equation*}
    \lambda_{n}(W_{A}(\Psi))=1 \quad \text{\rm if } \sum_{k=1}^{\infty}\prod_{i=1}^{n}\left(\frac{\psi_{i}(u_{k})}{\rho_{i}(u_{k})} \right) =\infty.
\end{equation*}
Lastly, observe that
\begin{align*}
    \sum_{k=1}^{\infty}2^{m\ell_{k}}\prod_{i=1}^{n}\psi_{i}(2^{\ell_{k}})& =\varepsilon^{n}c_{2}^{n}\sum_{k=1}^{\infty}\varepsilon^{-n}c_{2}^{-(n+m)}(c_{2}2^{\ell_{k}})^{m}\prod_{i=1}^{n}\psi_{i}(2^{\ell_{k}})\\
    &=\varepsilon^{n}c_{2}^{n}\sum_{k=1}^{\infty}\prod_{i=1}^{n}\frac{\psi_{i}(2^{\ell_{k}})}{\rho_{i}(u_{k})}\\
    & \geq \varepsilon^{n}c_{2}^{n}\sum_{k=1}^{\infty}\prod_{i=1}^{n}\frac{\psi_{i}(u_{k})}{\rho_{i}(u_{k})}
\end{align*}
 where the last line follows by the monotonicity of $\psi_{i}$ and that $c_{2}\geq 1$. Hence
 \begin{equation*}
     \sum_{k=1}^{\infty}2^{m\ell_{k}}\prod_{i=1}^{n}\psi_{i}(2^{\ell_{k}}) = \infty
 \end{equation*}
is sufficient to prove full measure of $W_{A}(\Psi)$. \par
For the convergence case of Corollary~\ref{bad_corollary}, apply the Borel-Cantelli Lemma to see that
\begin{equation*}
    \lambda_{n}(W_{A}(\Psi))=0 \quad \text{\rm if } \quad \sum_{r\in\N} \sum_{2^{r} < |\bq|_{\alpha}\leq 2^{r+1}} \lambda_{n}\left(\Delta(A\bq, \Psi(|\bq|_{\alpha}))\right)<\infty.
\end{equation*}
Observe that
\begin{align*}
    \sum_{r\in\N} \sum_{2^{r} < |\bq|_{\alpha}\leq 2^{r+1}} \lambda_{n}\left(\Delta(\{A\bq\}, \Psi(|\bq|_{\alpha}))\right)
    &\leq \sum_{r\in\N} c2^{rm} 2^{n}\prod_{i=1}^{n}\psi_{i}(2^{r})\, ,\\
    & \leq c 2^{n+1} \sum_{r=1}^{\infty}r^{m-1}\prod_{i=1}^{n}\psi_{i}(r),
\end{align*}
for constant $c=m2^{m+1}3^{m-1}\left(1-2^{-\min_{i} \alpha_{i}}\right)$ independent of $r$, and so the convergence summation condition of Corollary~\ref{bad_corollary} implies zero measure.

\subsection{Proof of Theorem~\ref{dim result}-\ref{dim result 2}}
For the upper bound consider the standard cover
\begin{equation*}
    \bigcup_{\bq\in\Z^{m}\backslash\{0\}: |\bq|_{\alpha}\geq N}\Delta\left(\{A\bq\}, \Psi(\bq)\right)
\end{equation*}
of $W_{A}(\Psi)$. Observe that each rectangle $\Delta\left(\{A\bq\}, \Psi(\bq)\right)$ can be covered by
\begin{equation*}
    \prod_{i=1 }^{n}\max\left\{1, \frac{|\bq|_{\alpha}^{-\tau_{i}}}{|\bq|_{\alpha}^{-\tau_{j}}}\right\}
\end{equation*}
balls of radius $|\bq|_{\alpha}^{-\tau_{j}}$.
%$|\bq|_{\alpha}^{\max\{\tau_{1},\tau_{2}\}-\min\{\tau_{1},\tau_{2}\}}$ balls of radius $|\bq|_{\alpha}^{-\max\{\tau_{1},\tau_{2}\}}$.
Thus
\begin{align*}
    \cH^{s}(W_{A}(\Psi)) & \leq \sum_{\bq \in \Z^{m}:|\bq|_{\alpha}\geq N} \prod_{i=1 }^{n}\max\left\{1, \frac{|\bq|_{\alpha}^{-\tau_{i}}}{|\bq|_{\alpha}^{-\tau_{j}}}\right\} |\bq|_{\alpha}^{-s\tau_{j}}\, , \\
    & = \sum_{r\in\N_{\geq N}} \, \sum_{\bq \in \Z^{m} : |\bq|_{\alpha}=r} r^{\sum_{i:\tau_{j}>\tau_{i}}(\tau_{j}-\tau_{i})-s\tau_{j}}\, , \\
    & \leq 2^{m} \sum_{r\in\N_{\geq N}} \,  r^{m-1+\sum_{i:\tau_{j}>\tau_{i}}(\tau_{j}-\tau_{i})-s\tau_{j}} \to 0
\end{align*}
as $N \to \infty$ for any 
\begin{equation*}
s>\frac{m+\sum_{i:\tau_{j}>\tau_{i}}(\tau_{j}-\tau_{i})}{\tau_{j}}.    
\end{equation*}
This argument is true for any choice of $1\leq j \leq n$, hence 
\begin{equation*}
    \dimh W_{A}(\Psi) \leq \min_{1\leq i \leq n}\left\{  \frac{m+\sum_{i:\tau_{j}>\tau_{i}}(\tau_{j}-\tau_{i})}{\tau_{j}}\right\},
\end{equation*}
completing the upper bound result. \par
For the lower bound dimension result we appeal to Theorem~\ref{MTPRR}.
Given the setup provided in Theorem~\ref{MTPRR} we have that
\begin{equation*}
    a_{i}=\frac{m}{n}\, , \quad t_{i}=\tau_{i}-\frac{m}{n} \quad (1\leq i \leq n)\, , \quad \rho(r)=\left(\varepsilon c_{2}^{1+\frac{m}{n}}\right)^{\frac{n}{m}}r^{-1}\, .
\end{equation*}
For ease of notation, let
\begin{equation*}
    d(A)=\sum_{j \in \cK_{1}} 1+ \sum_{j \in \cK_{2}} 1+ \frac{\sum_{j \in \cK_{3}}a_{j}-\sum_{j \in \cK_{2}}t_{j}}{A} 
\end{equation*}
with each set $\cK_{1},\cK_{2}$ and $\cK_{3}$ defined by $A$.
Consider the two cases:
\begin{itemize}
    \item $A=\frac{m}{n}$: Then the sets appearing in Theorem~\ref{MTPRR} are
    \begin{equation*}
        \cK_{1}=\{ 1,\dots, n\}\, ,\quad \cK_{2}=\emptyset\, , \quad \cK_{3}=\emptyset\,
    \end{equation*}
    and so $d(A)=n$, giving us a trivial lower bound.
    \item $A=\tau_{j}$ for some $1\leq j \leq n$: Then
    \begin{equation*}
        \cK_{1}=\emptyset\, ,\quad \cK_{2}=\{ i: \tau_{i}\leq \tau_{j}\}\, , \quad \cK_{3}=\{i: \tau_{i}>\tau_{j}\}=\{1,\dots,n\}\backslash \cK_{2}\,
    \end{equation*}
    and so
    \begin{align*}
        d(A)&=\#\cK_{2}+\frac{\frac{m}{n}(n-\#\cK_{2})-\sum_{i: \tau_{i}\leq \tau_{j}}\left(\tau_{i}-\frac{m}{n}\right)}{A} \\
        &=\frac{m+ (A-\frac{m}{n})\#\cK_{2} -\sum_{i\in\cK_{2}}\left(\tau_{i}-\frac{m}{n}\right)}{A} \\
        &=\frac{m+\sum_{i\in\cK_{2}}\left(A-\tau_{i}\right)}{A}\, . \\
    \end{align*}
    Hence, replacing $A$ by some $\tau_{j}$ and inputting the definition of $\cK_{2}$ we have that
    \begin{equation*}
        d(\tau_{j})=\frac{m+\sum_{i:\tau_{i}\leq \tau_{j}}\left(\tau_{j}-\tau_{i}\right)}{\tau_{j}}
    \end{equation*}
\end{itemize}
Taking the minimum over all possible choices of $1\leq j \leq n$ we have the desired lower bound. The measure result is simply due to the second part of Theorem~\ref{MTPRR} and the observation that $s<n$.

For the lower bound of Theorem~\ref{dim result 2} let
\begin{equation*}
    a_{i}=v_{i}\frac{m}{n}\, , \quad t_{i}=\tau_{i}-v_{i}\frac{m}{n} \quad (i=1,2) \, \quad \rho(r)=\left(\varepsilon c_{2}^{1+\frac{m}{n}}\right)^{\frac{n}{m}}r^{-1}\, .
\end{equation*}
Without loss of generality we assume $v_{1}<v_{2}$.
\begin{itemize}
    \item $v_{1}\frac{m}{n}<v_{2}\frac{m}{n}<\tau_{1}<\tau_{2}$: Consider the sets $\cK_{1},\cK_{2}$ and $\cK_{3}$, and the corresponding dimension bound for each of the following
    \begin{itemize}
        \item $A=v_{1}\frac{m}{n}$:
        \begin{equation*}
            \cK_{1}=\{ 1,2\}\, ,\quad \cK_{2}=\emptyset\, , \quad \cK_{3}=\emptyset\, .
        \end{equation*}
        So 
        \begin{equation*}
        d(A)= 2.
        \end{equation*}
\item $A=v_{2}\frac{m}{n}$:
        \begin{equation*}
            \cK_{1}=\{ 2\}\, ,\quad \cK_{2}=\emptyset\, , \quad \cK_{3}=\{1\}\, .
        \end{equation*}
        So 
        \begin{equation*}
        d(A)= 1+\frac{v_{1}\frac{m}{n}}{v_{2}\frac{m}{n}}=1+\frac{v_{1}}{v_{2}}.
        \end{equation*}
        \item $A=\tau_{1}$:
        \begin{equation*}
            \cK_{1}=\emptyset\, ,\quad \cK_{2}=\{ 1\}\, , \quad \cK_{3}=\{ 2\}\, .
        \end{equation*}
        So 
        \begin{equation*}
        d(A)=1+\frac{v_{1}\frac{m}{n}+v_{2}\frac{m}{n}-\tau_{1}}{\tau_{1}}=1+\frac{m-\tau_{1}}{\tau_{1}}.
        \end{equation*}

        \item $A=\tau_{2}$:
        \begin{equation*}
            \cK_{1}=\emptyset\, ,\quad \cK_{2}=\{ 1,2\}\, , \quad \cK_{3}=\emptyset\, .
        \end{equation*}
        So 
        \begin{equation*}
        d(A)=2-\frac{\left(\tau_{1}-v_{1}\frac{m}{n}\right)+\left(\tau_{2}-v_{2}\frac{m}{n}\right)}{\tau_{2}}=1+\frac{m-\tau_{1}}{\tau_{2}}.
        \end{equation*}
    \end{itemize}
     Since $\tau_{2}>\tau_{1}$ we have, in this case, that
    \begin{equation*}
        \min d(A)= 1+\min\left\{ \frac{m-\tau_{1}}{\tau_{2}}, \frac{v_{1}}{v_{2}}\right\}.
    \end{equation*}

    \item $v_{1}\frac{m}{n}<v_{2}\frac{m}{n}<\tau_{2}<\tau_{1}$: This is similar to the previous case with $\tau_{1}$ and $\tau_{2}$ switching roles. In particular we get
    \begin{equation*}
        \min d(A)= 1+\min\left\{\frac{v_{1}}{v_{2}}, \frac{m-\tau_{2}}{\tau_{1}} \right\}.
    \end{equation*}
 \end{itemize}
 Combining these cases together, and using the condition that
 \begin{equation*}
     \frac{v_{1}}{v_{2}}\geq\frac{m-\min\{\tau_{1},\tau_{2}\}}{\max\{\tau_{1},\tau_{2}\}}
 \end{equation*}
   we obtain our lower bound. Lastly, consider the following case:
\begin{itemize}
    \item $v_{1}\frac{m}{n}<\tau_{1}<v_{2}\frac{m}{n}<\tau_{2}$: Consider the sets $\cK_{1},\cK_{2}$ and $\cK_{3}$, and the corresponding dimension bound for each of the following
    \begin{itemize}
        \item $A=v_{1}\frac{m}{n}$:
        \begin{equation*}
            \cK_{1}=\{1,2 \}\, ,\quad \cK_{2}=\emptyset\, , \quad \cK_{3}=\emptyset\, .
        \end{equation*}
        So
        \begin{equation*}
            d(A)=2.
        \end{equation*}

        \item $A=v_{2}\frac{m}{n}$:
        \begin{equation*}
            \cK_{1}=\{2 \}\, ,\quad \cK_{2}=\{1\}\, , \quad \cK_{3}=\emptyset\, .
        \end{equation*}
        So
        \begin{equation*}
            d(A)=2-\frac{\tau_{1}-v_{1}\frac{m}{n}}{v_{2}\frac{m}{n}}=1+\frac{v_{1}\frac{m}{n}+v_{2}\frac{m}{n}-\tau_{1}}{v_{2}\frac{m}{n}}=1+\frac{m-\tau_{1}}{v_{2}\frac{m}{n}}.
        \end{equation*}

        \item $A=\tau_{1}$:
        \begin{equation*}
            \cK_{1}=\{ 2\}\, ,\quad \cK_{2}=\{ 1\}\, , \quad \cK_{3}=\emptyset\, .
        \end{equation*}
        So,
        \begin{equation*}
            d(A)=2-\frac{\tau_{1}-v_{1}\frac{m}{n}}{\tau_{1}}=1+\frac{v_{1}\frac{m}{n}}{\tau_{1}}.
        \end{equation*}

        \item $A=\tau_{2}$:
        \begin{equation*}
            \cK_{1}=\emptyset\, ,\quad \cK_{2}=\{ 1,2\}\, , \quad \cK_{3}=\emptyset\, .
        \end{equation*}
        So
        \begin{equation*}
            d(A)=2-\frac{\left(\tau_{1}-v_{1}\frac{m}{n}\right)+\left(\tau_{2}-v_{2}\frac{m}{n}\right)}{\tau_{2}}=1+\frac{m-\tau_{1}}{\tau_{2}}.
        \end{equation*}
    \end{itemize}
    In this case, noting $\tau_{2}>v_{2}\frac{m}{n}$, we have that
    \begin{equation*}
        \min d(A)= 1+\min\left\{ \frac{v_{1}\frac{m}{n}}{\tau_{1}}, \frac{m-\tau_{1}}{\tau_{2}}\right\}.
    \end{equation*}
\end{itemize}
Using \eqref{annoying condition}, this minimum again becomes our desired lower bound.

\bibliographystyle{plain}
\bibliography{bibliography4.bib}

\begin{thebibliography}{10}

\bibitem{BengMosch2017}
P~Bengoechea and N~Moshchevitin.
\newblock Badly approximable points in twisted {D}iophantine approximation and
  {H}ausdorff dimension.
\newblock {\em Acta Arith.}, 177(4):301--314, 2017.

\bibitem{BDGW23}
V.~Beresnevich, S.~Datta, A.~Ghosh, and B.~Ward.
\newblock Bad is null.
\newblock {\em Preprint: arXiv:2307.10109}, 2023.

\bibitem{BDGW2}
V.~Beresnevich, S.~Datta, A.~Ghosh, and B.~Ward.
\newblock Rectangular shrinking targets for $\mathbb{Z}^m$ actions on tori:
  well and badly approximable systems.
\newblock {\em Preprint: arXiv:2307.10122}, 2023.

\bibitem{BDV}
Victor Beresnevich, Detta Dickinson, and Sanju Velani.
\newblock Measure theoretic laws for lim sup sets.
\newblock {\em Mem. Amer. Math. Soc.}, 179:no. 846, x+91 pp., 2006.

\bibitem{BernikDodson}
V.~I. Bernik and M.~M. Dodson.
\newblock {\em Metric {D}iophantine approximation on manifolds}, volume 137 of
  {\em Cambridge Tracts in Mathematics}.
\newblock Cambridge University Press, Cambridge, 1999.

\bibitem{Bugeaud2003}
Yann Bugeaud.
\newblock A note on inhomogeneous {D}iophantine approximation.
\newblock {\em Glasg. Math. J.}, 45(1):105--110, 2003.

\bibitem{Cassels}
John W.~S. Cassels.
\newblock {\em An introduction to {D}iophantine approximation}.
\newblock Cambridge Tracts in Mathematics and Mathematical Physics, No. 45.
  Cambridge University Press, New York, 1957.

\bibitem{FanWu06}
Ai-Hua Fan and Jun Wu.
\newblock A note on inhomogeneous {D}iophantine approximation with a general
  error function.
\newblock {\em Glasg. Math. J.}, 48(2):187--191, 2006.

\bibitem{FuchsKim}
Michael Fuchs and Dong~Han Kim.
\newblock On {K}urzweil's 0-1 law in inhomogeneous {D}iophantine approximation.
\newblock {\em Acta Arith.}, 173(1):41--57, 2016.

\bibitem{Harrap12}
Stephen Harrap.
\newblock Twisted inhomogeneous {D}iophantine approximation and badly
  approximable sets.
\newblock {\em Acta Arith.}, 151(1):55--82, 2012.

\bibitem{MosHar}
Stephen Harrap and Nikolay Moshchevitin.
\newblock A note on weighted badly approximable linear forms.
\newblock {\em Glasg. Math. J.}, 59(2):349--357, 2017.

\bibitem{Khintchine1946}
A.~Khintchine.
\newblock Sur le probl\`eme de {T}chebycheff.
\newblock {\em Bull. Acad. Sci. URSS. S\'{e}r. Math. [Izvestia Akad. Nauk
  SSSR]}, pages 281--294, 1946.

\bibitem{Kim07}
Dong~Han Kim.
\newblock The shrinking target property of irrational rotations.
\newblock {\em Nonlinearity}, 20(7):1637--1643, 2007.

\bibitem{Kim14}
Dong~Han Kim.
\newblock Refined shrinking target property of rotations.
\newblock {\em Nonlinearity}, 27(9):1985--1997, 2014.

\bibitem{Kim23}
Taehyeong Kim.
\newblock On a kurzweil type theorem via ubiquity.
\newblock {\em Preprint: arXiv:2306.00847}, 2023.

\bibitem{KW23}
Dmitry Kleinbock and Baowei Wang.
\newblock Measure theoretic laws for limsup sets defined by rectangles.
\newblock {\em Adv. Math.}, 428:Paper No. 109154, 2023.

\bibitem{Kur55}
Jaroslav Kurzweil.
\newblock On the metric theory of inhomogeneous diophantine approximations.
\newblock {\em Studia Math.}, 15:84--112, 1955.

\bibitem{LRSY}
A.~Shapira U. Yifrach~Y. Lachman, G.~Rao.
\newblock K-divergent lattices.
\newblock {\em Preprint: arXiv:2307.09054}, 2023.

\bibitem{Minkowski1900}
Hermann Minkowski.
\newblock Ueber die {A}nn\"{a}herung an eine reelle {G}r\"{o}sse durch
  rationale {Z}ahlen.
\newblock {\em Math. Ann.}, 54(1-2):91--124, 1900.

\bibitem{MRS}
A.~Shapira~U. Moshchevitin, N.~Rao.
\newblock Grids with dense values 2.
\newblock In preparation.

\bibitem{Ramirez18}
Felipe~A. Ram\'{\i}rez.
\newblock A characterization of bad approximability.
\newblock {\em Nonlinearity}, 31(9):4246--4262, 2018.

\bibitem{Simmons15}
David Simmons.
\newblock An analogue of a theorem of {K}urzweil.
\newblock {\em Nonlinearity}, 28(5):1401--1408, 2015.

\bibitem{SchmTrot03}
S.~Troubetsko\u{\i} and \u{I}. Schmeling.
\newblock Inhomogeneous {D}iophantine approximations and angular recurrence for
  billiards in polygons.
\newblock {\em Mat. Sb.}, 194(2):129--144, 2003.

\bibitem{Tseng08}
Jimmy Tseng.
\newblock On circle rotations and the shrinking target properties.
\newblock {\em Discrete Contin. Dyn. Syst.}, 20(4):1111--1122, 2008.

\bibitem{Ravenstein1988}
Tony van Ravenstein.
\newblock The three gap theorem ({S}teinhaus conjecture).
\newblock {\em J. Austral. Math. Soc. Ser. A}, 45(3):360--370, 1988.

\bibitem{WW19}
Baowei Wang and Jun Wu.
\newblock Mass transference principle from rectangles to rectangles in
  {D}iophantine approximation.
\newblock {\em Math. Ann.}, 381(1-2):243--317, 2021.

\end{thebibliography}

\end{document}